\documentclass{amsart}
\usepackage[utf8]{inputenc}
\usepackage{amssymb}
\usepackage{amsmath}
\usepackage{amsthm}
\usepackage{bbm}
\usepackage{calrsfs}
\usepackage{csquotes}
\usepackage{geometry}
\usepackage{hyperref}
\usepackage{mathtools}
\theoremstyle{definition}
\newtheorem{definition}{Definition}[section]
\newtheorem{theorem}{Theorem}[section]
\newtheorem{corollary}{Corollary}[theorem]
\newtheorem{lemma}[theorem]{Lemma}
\newtheorem{remark}[theorem]{Remark}
\newtheorem{proposition}[theorem]{Proposition}

\newcommand{\stirlingone}[2]{\genfrac{[}{]}{0pt}{}{#1}{#2}}

\DeclarePairedDelimiter\floor{\lfloor}{\rfloor}

\DeclareMathOperator{\sign}{sign}
\DeclareMathOperator{\Tr}{Tr}

\title{A convergent $1/n$-expansion for GSE and GOE}
\author{Yaroslav Naprienko}
\thanks{\href{mailto:ya.naprienko@gmail.com}{ya.naprienko@gmail.com}, Department of Mathematics, Higher School of Economics, Russian Federation}
\begin{document}

\maketitle
\begin{abstract}
   We generalize the results on the asymptotic expansion from Gaussian Unitary Ensembles case to all Gaussian Ensembles. We derive differential equations on densities and their moment generating functions for all Gaussian Ensembles. Also, we calculate explicit expressions for the moment generating functions for all Gaussian Ensembles.
\end{abstract}

\numberwithin{equation}{section}

\section{Main results}
The aim of this paper is to find convergent asymptotic expansion for Gaussian Symplectic and Orthogonal Ensembles analogous to Offer Kopelevitch's expansion for Gaussian Unitary Ensemble. We prove following theorems:
\begin{theorem}
    Let $g$ be an entire function of order $\rho_g \leq 2$ and type $\sigma_g$ with $\sigma_g<+\infty$ if $\rho_g=2$. Then for a GSE random matrix $X_n$ of order $n \ge n_{s}(\rho_g,\sigma_g)$ and the scaled mean density function $\hat{p}_s(x)$, one has the following convergent expansion:
    \begin{equation}
        \int_{\mathbb{R}} g(x)\hat{p}_{s}(x)dx = \frac{3}{4\pi}\sum_{j=0}^{\infty}\frac{1}{n^{2j}}\int_2^2 \biggl[T^j \biggl(\frac{1}{n}f'-xf\biggr)\biggr](t)\sqrt{4-t^2}dt,
    \end{equation}
    where $T:C_b^{\infty} \to C_b^{\infty}$ is a linear operator defined in Secton \ref{SectionAsymptotic} and $f$ is any solution of the differential equation
    \begin{equation}
        g(x) = -\frac{4}{n^2}f'''(x) + \left(x^2 - \frac{8n+2}{n}\right)f'(x).
    \end{equation}
\end{theorem}

\begin{theorem}
    Let $g$ be an entire function of order $\rho_g \leq 2$ and type $\sigma_g$ with $\sigma_g<+\infty$ if $\rho_g=2$. Then for a GOE random matrix $X_n$ of order $n \ge n_{o}(\rho_g,\sigma_g)$ and the scaled mean density function $\hat{p}_o(x)$, one has the following convergent expansion:
    \begin{equation}
        \int g(x)\hat{p}_{o}dx = \frac{3}{4\pi}\sum_{j=0}^{\infty}\frac{1}{n^{2j}}\int_2^2 \biggl[T^j \biggl(\frac{1}{n}f'-xf\biggr)\biggr](t)\sqrt{4-t^2}dt,
    \end{equation}
    where $T:C_b^{\infty} \to C_b^{\infty}$ is a linear operator defined in Secton \ref{SectionAsymptotic} and $f$ is any solution of the differential equation
    \begin{equation}
        g(x) = -\frac{4}{n^2}f'''(x) + \left(x^2 - \frac{4n-2}{n}\right)f'(x).
    \end{equation}
\end{theorem}

Also, we find differential equations on densities of Gaussian Symplectic and Orthogonal Ensembles. 
\begin{proposition}
The density of GSE (\ref{GSEdensity}) satisfies the differential equation
\begin{multline}
    4\sigma^4 p_{GSE,n,\sigma^2}'''(x) + (\sigma^2(8n+2)- x^2)p_{GSE,n,\sigma^2}'(x) - 2x p_{GSE,n,\sigma^2}(x) = \\
    =-3\sigma^2 p_{GUE,n,\sigma^2}'-3x p_{GUE,n,\sigma^2}.
\end{multline}
\end{proposition}
\begin{proposition}
The density of GSE (\ref{GOEdensity}) satisfies the differential equation
\begin{multline}
    4\sigma^4 p_{GOE,n,\sigma^2}'''(x) + (\sigma^2(4n-2)- x^2)p_{GOE,n,\sigma^2}'(x) - 2x p_{GOE,n,\sigma^2}(x) = \\
    =-3\sigma^2 p_{GUE,n,\sigma^2}'-3x p_{GUE,n,\sigma^2}.
\end{multline}
\end{proposition}

\section{Introduction and Preliminaries}
Recall the following well-known result on asymptotics for distribution of eigenvalues for random matrices.

\begin{definition}
    A Wigner matrix $X_n = (X_{ij})_{i,j=1}^n$ is a random $n \times n$ matrix such that,\\
    (i)~$X_n$ is self-adjoint: $X_{ij}=\overline{X_{ji}}$,\\
    (ii)~$X_{ij}$ with $i\le j$ are independent,\\
    (iii)~$X_{ij}$ with $i<j$ are identically distributed with zero mean and variance equal to~$1$,\\
    (iv)~$X_{ii}$ are identically distributed with bounded mean and variance.
\end{definition}

\begin{theorem}[Wigner's Semicircle Law]\label{thm:Wigner}
    Let $X_n$ be a $n \times n$ Wigner matrix. Then scaled empirical spectral distribution of eigenvalues $\lambda_j(X_n)$ converges almost surely (hence also in probability and in expectation) to the Wigner Semicircle distribution:
    \begin{equation*}
        \frac{1}{n}\sum_{j=1}^n \delta_{\lambda_j(X_n)/\sqrt{n}} \xrightarrow[n \to \infty]{w} \mathbbm{1}_{[-2;2]}\frac{1}{2\pi}\sqrt{4-x^2}dx.
    \end{equation*}
\end{theorem}
In special cases it is possible to compute the distribution of the eigenvalues explicitly, for example, in the case of Gaussian Ensembles.
\begin{definition}
    \emph{Gaussian Orthogonal Ensemble} (\emph{GOE}), \emph{Gaussian Unitary Ensemble} (\emph{GUE}), and \emph{Gaussian Symplectic Ensemble} (\emph{GSE}) are ensembles of random matrices satisfying conditions~(i) and~(ii) of Theorem~\ref{thm:Wigner} such that
    for any $i<j$  one has $X_{ij}\sim N(0,2\sigma^2)_{\mathbb{R}}$ for GOE, $X_{ij}\sim N(0,\sigma^2)_{\mathbb{C}}$ for GUE, 
    $X_{ij}\sim N(0,\sigma^2 / 2)_{\mathbb{H}}$ for GSE, and $X_{ii}\sim N(0,\sigma^2)_{\mathbb{R}}$ for all three ensembles.
\end{definition}

If $\sigma^2 = 1/2$ all these ensembles satisfy Wigner's Semicircle Law.

Let $p_n$ be \emph{mean density} for distribution of eigenvalues for one of these ensembles:
 \begin{equation*}
    \frac{1}{n}\mathbb{E}\sum_{j=1}^n \mathbbm{1}_{[a,b]}(\lambda_j)=\int_a^b p_n(x)dx,
\end{equation*}
The functions $p_n(x)$ can be calculated explicitly, see \cite{M} and formulas (\ref{GUEdensity}), (\ref{GSEdensity}) and (\ref{GOEdensity}) below.

U.~Haagerup and S.~Thorbjornsen~\cite{HT2} derived the explicit expression for the moment generating function (the bilateral Laplace transform) of the GUE density in terms of confluent hypergeometric functions. F.~Gotze and A.~Tikhomirov~\cite{G} used this expression to find the differential equation (\ref{GUEdiff}) on the GUE density. Using this differential equation U.~Haagerup and S.~Thorbjornsen found asymptotic expansion (\ref{GUEasymp}) as a power series in~$1/n$. They did not deal with the convergence of the series. O. Kopelevitch showed that this series is convergent for a certain class of entire functions and that this class is ``sharp'' in some sense (see remark after Theorem~1 in \cite{K}). M.~Ledoux~\cite{L} found differential equations on the moment generating functions for all Gaussian Ensembles, and used them to find a recursive formula for the moments of these ensembles.

The aim of this paper is to find convergent asymptotic expansion for GOE and GSE analogous to Kopelevitch's expansion for GUE.

The main steps to do this are the following. In Section \ref{SectionDifferential} we derive differential equations on densities for Gaussian Ensembles. Note that in the GUE case we do not use the explicit calculation of the moment generating function but give a new proof. In Section \ref{SectionAsymptotic} we use these differential equations to obtain the convergent expansions for GSE and GOE by reduction to Kopelevitch's theorem~\cite{K}.

In Section \ref{SectionDifferentialMoments} we derive differential equations on the moment generating functions of densities for Gaussian Ensembles. These results were obtained by M.~Ledoux~\cite{L} but we give an alternative proof. In works of U.~Haagerup, S.~Thorbjornsen~\cite{HT} and M.~Ledoux~\cite{L} a question on explicit expressions of the moment generating functions for matrices from GSE or GOE was posed: 

\begin{displayquote}
"It would be interesting to know the counterparts of the explicit formulas (0.2) [\textit{the moment generating function}] for random matrices with real or symplectic Gaussian entries. The real and symplectic counterparts of the [GUE] density are computed in Mehta's book ... [\eqref{GSEdensity} and \eqref{GOEdensity}] ... However, the formulas for these densities are much more complicated than in the complex case."~\cite{HT}
\end{displayquote}

In Section \ref{SectionMoment} we calculate these expressions in terms of confluent hypergeometric functions as was supposed by M.~Ledoux~\cite{L} and in Section \ref{SectionMomentExpansion} we give their explicit $1/n$-expansions.

Recall that the \emph{confluent hypergeometric function} is defined as follows:
\begin{equation}\label{Hyper}
    {_1F_1}(a,b,x) := \sum_{j=0}^{\infty}\frac{a(a+1)\dots(a+j-1)}{b(b+1)\dots(b+j-1)}\frac{x^j}{j!}
\end{equation}
for $a,b,x \in \mathbb{C}$ such that $b \not\in \mathbb{Z} \setminus \mathbb{N}$. The function ${_1F_1}(a,b,x)$ is an entire function of $x$ (a polynomial in the case $a \in \mathbb{Z} \setminus \mathbb{N}$) that satisfies the differential equation
\begin{equation*}
  x y'' + (b-x)y' - ay = 0.
\end{equation*}
The \emph{Hermite functions} $\phi_k(x)$ and \emph{Hermite polynomials} $H_k(x)$ are defined as follows:
\begin{equation*}
    \phi_k(x) := \frac{1}{\sqrt{2^k k! \sqrt{\pi}}}H_k(x)e^{-\frac{x^2}{2}},\quad 
    H_k(x) := (-1)^k e^{x^2} \times \left( \frac{d^k}{dx^k}e^{-x^2} \right).
\end{equation*}
Then the mean density functions for GUE, GSE and GOE are given by the formulas (see \cite{M})
\begin{align}\label{GUEdensity}
    p_{GUE,n,\sigma^2}(x)&{}= \frac{1}{\sigma\sqrt{2}}\sum_{k=0}^{n-1}\phi_k\left(\frac{x}{\sigma\sqrt{2}}\right)^2,\\
    \label{GSEdensity}
    p_{GSE,n,\sigma^2}(x)&{} = p_{GUE,2n+1,\sigma^2}(x)+\frac{1}{\sigma\sqrt{2}}\tau_{2n+1}\left(\frac{x}{\sigma \sqrt{2}}\right),\\
    \label{GOEdensity}
    p_{GOE,n,\sigma^2}(x)&{} = p_{GUE,n,\sigma^2}(x)+\frac{1}{\sigma\sqrt{2}}\left(\tau_{n}\left(\frac{x}{\sigma \sqrt{2}}\right)+\alpha_{n}\left(\frac{x}{\sigma \sqrt{2}}\right)\right),
\end{align}
where
\begin{equation}\label{Taufunction}
    \tau_{n}(x) = \sqrt{\frac{n}{2}}\phi_{n-1}(x)\frac{1}{2}\int_{\mathbb{R}} \sign\left(x-t\right)\phi_{n}(t)dt,
\end{equation}
\begin{equation}\label{Alphafunction}
    \alpha_n(x) = \begin{cases}\phi_{n-1}(x)\left(\int_{\mathbb{R}}\phi_{n-1}(t)dt)\right)^{-1},&n\text{ is odd},\\
    0,&n\text{ is even.}\end{cases}
\end{equation}
\section{Differential equations on densities}\label{SectionDifferential}
Here we obtain differential equations on densities of Gaussian ensembles. In this section we use the following notation:
\begin{align*}
    p_{u}(x)&{}= p_{GUE,n,1/2}(x)= \sum_{k=0}^{n-1}\left[\phi_k(x)\right]^2,\\
    p_{s}(x)&{}= p_{GSE,n,1/2}(x) = p_u(x)+\tau_{2n+1}(x),\\
    p_{o}(x)&{}= p_{GOE,n,1/2}(x)= p_u(x) + \tau_{n}(x) + \alpha_n(x).
\end{align*}
\subsection{GUE case}
F.~Gotze and A.~Tikhomirov in \cite{G} derived the following theorem. For completeness, we provide here an alternative proof of their result.  
\begin{theorem}[F.~Gotze, A.~Tikhomirov, \cite{G}, Lemma 2.1]\label{theoremGUEdiff}
The density of GUE (\ref{GUEdensity}) satisfies the differential equation
\begin{equation}\label{GUEdiff}
    \sigma^4 p_{GUE,n,\sigma^2}'''(x) + (4n\sigma^2- x^2)p_{GUE,n,\sigma^2}'(x) + x p_{GUE,n,\sigma^2}(x) = 0.
\end{equation}
\end{theorem}
\begin{proof}
    By changing of variable $x \mapsto \sqrt{2}\sigma x$ we reduce the general case to the $\sigma^2=1/2$ case.
    
    One can show that (see \eqref{DerivSumSq})
    \begin{equation}\label{eq:HermIdent1}
        p_{u}' = -\sqrt{2n}\phi_n\phi_{n-1},
    \end{equation}
    and, using identities
    \begin{equation}\label{eq:HermIdent2}
        \phi_n''(x) = (x^2 - 2n - 1)\phi_n,\quad 
        2x p_u(x) = \sqrt{2n}\left((2n-x^2)\phi_n\phi_{n-1}+\phi_n'\phi_{n-1}'\right)
    \end{equation}
    that
    \begin{equation*}
        p_{u}''' = -\sqrt{2n}\left(\phi_n''\phi_{n-1} + 2\phi_n'\phi_{n-1}' + \phi_n\phi_{n-1}''\right)=
        -4(2n- x^2)p_{u}' - 4 x p_{u}.\qedhere
    \end{equation*}
\end{proof}

\subsection{GSE case}
\begin{lemma}
The function $\tau_n(x)$ defined in \eqref{Taufunction} satisfies the differential equation
\begin{equation}\label{Taudiff}
    2\tau_n''' + 2(2n - 1 - x^2)\tau_n' - 4x\tau_n = (12n-1-6x^2)p_u' + 6xp_u.
\end{equation}
\end{lemma}
\begin{proof}
    Using identities \eqref{eq:HermIdent1}, \eqref{eq:HermIdent2} we get
    \begin{gather*}
       2\tau_n'''(x) = \sqrt{2n}\left(\phi_{n-1}''\frac{1}{2}\int_{\mathbb{R}}\sign(x-t)\phi_n(t)dt+\phi_{n-1}'\phi_n+\phi_{n-1}\phi_n'\right)',\\
        2\tau_n''' + 2\left((2n - 1 - x^2)\tau_n\right)' = \sqrt{2n}\left(2\phi_{n-1}''\phi_n+\phi_{n-1}\phi_n''+3\phi_n'\phi_{n-1}'\right),
    \end{gather*}
    and finally~\eqref{Taudiff}.
\end{proof}

\begin{proposition}\label{theoremGSEdiff}
The density of GSE (\ref{GSEdensity}) satisfies the differential equation
\begin{multline}\label{GSEdiff}
    4\sigma^4 p_{GSE,n,\sigma^2}'''(x) + (\sigma^2(8n+2)- x^2)p_{GSE,n,\sigma^2}'(x) - 2x p_{GSE,n,\sigma^2}(x) = \\
    =-3\sigma^2 p_{GUE,n,\sigma^2}'-3x p_{GUE,n,\sigma^2}.
\end{multline}
\end{proposition}
\begin{proof}
    Here $p^{*}_u(x) = p_{GUE,2n+1,1/2}(x)$.
    \begin{equation*}
        2p_s = 2p^{*}_u + 2\tau_{2n+1}.
    \end{equation*}
    Combine (\ref{GUEdiff}) and (\ref{Taudiff}) for $2n+1$:
    \begin{multline*}
        2p_s''' + 2(4n+1-x^2)p_s' - 4xp_s = \\
        = \left(2p^{*\prime\prime\prime}_u + 2(4n+1-x^2)p^{*\prime}_u - 4xp^{*}_u\right) + \left(2\tau_{2n+1}''' + 2(4n+1-x^2)\tau_{2n+1}' - 4x\tau_{2n+1}\right) = \\ 
        = \left((6x^2-24n-14)p^{*\prime}_u - 12xp^{*}_u\right) + \left((24n+11-6x^2)p^{*\prime}_u + 6xp^{*}_u\right) = -3p^{*\prime}_u - 6xp^{*}_u.
    \end{multline*}
    Now (\ref{GSEdiff}) follows by change of variables $x \mapsto x/\sqrt{2}\sigma$.
\end{proof}
\subsection{GOE case}
\begin{proposition}\label{theoremGOEdiff}
The density of GSE (\ref{GOEdensity}) satisfies the differential equation
\begin{multline}\label{GOEdiff}
    4\sigma^4 p_{GOE,n,\sigma^2}'''(x) + (\sigma^2(4n-2)- x^2)p_{GOE,n,\sigma^2}'(x) - 2x p_{GOE,n,\sigma^2}(x) = \\
    =-3\sigma^2 p_{GUE,n,\sigma^2}'-3x p_{GUE,n,\sigma^2}.
\end{multline}
\end{proposition}
\begin{proof}
    Note that $\alpha_n''' + (2n-1-x^2)\alpha_n' - 2x\alpha_n = 0$ by the first identity in~\eqref{eq:HermIdent2}. Then 
    \begin{multline*}
            2p_o''' + 2(2n-1-x^2)p_o' - 4xp_o ={}\\
            {}=\left(2p_u''' + 2(2n-1-x^2)p_u' - 4xp_u\right) +  \left(2\tau_n''' + 2(2n-1-x^2)\tau_n' - 4x\tau_n\right)
        = -3p_u' - 6xp_u,
    \end{multline*}
    and (\ref{GOEdiff}) follows by change of variables $x \mapsto x/\sqrt{2}\sigma$.
\end{proof}

\begin{remark}
The differential equation for GSE case is, exactly, the differential equation for the GOE case of order $2n+1$.
\end{remark}

\section{Asymptotic expansion for Gaussian Ensembles}\label{SectionAsymptotic}

In this section we obtain asymptotic expansions for GSE and GOE from those for GUE.
The latter was found by U.~Haagerup and S.~Thorbjornsen \cite{HT2}, and its convergence was shown by O.~Kopelevitch, \cite{K}. Let us recall their results.

Here we will use the following notation: 
\begin{equation*}
    \hat{p}_{u}(x) = p_{GUE,n,1/n}(x), \quad 
    \hat{p}_{s}(x) = p_{GSE,n,1/n}(x), \quad 
    \hat{p}_{o}(x) = p_{GOE,n,1/n}(x).
\end{equation*}

\begin{proposition}[\cite{HT2}, Proposition 2.4]
    For any $C^{\infty}$-function $g: \mathbb{R} \to \mathbb{C}$, there is a unique $C^{\infty}$-function $f: \mathbb{R} \to \mathbb{C}$ such that
    \begin{equation}\label{HaagerupEquation}
        g(x) = \frac{1}{2\pi}\int_2^2 g(t)\sqrt{4-t^2}dt + (x^2 - 4)f'(x)+3x f(x),
    \end{equation}
    and if $g \in C_b^{\infty}$, then $f \in C_b^{\infty}$.
\end{proposition}

Let $S:C_b^{\infty} \to C_b^{\infty}$ be a linear operator defined by $Sg = f$, where $f$ is the unique solution to (\ref{HaagerupEquation}). Also let a linear operator $T:C_b^{\infty} \to C_b^{\infty}$ be defined by $Tg = (Sf)'''$.

\begin{theorem}[\cite{HT2}, Theorem 3.5]
    For any function $g \in C_b^{\infty}$ we have that
    \begin{equation}\label{HaagerupRecurrent}
        \int g(x)\hat{p}_{u}dx = \frac{1}{2\pi}\int_2^2 g(t)\sqrt{4-t^2}dt + \frac{1}{n^2}\int [Tg](x)\hat{p}_{u}dx.
    \end{equation}
\end{theorem}
\begin{proof}
    Multiply (\ref{GUEdiff}) for $\sigma^2 = 1/n$ by $f \in C_b^{\infty}$ and integrate over the real line:
    \begin{equation*}
        \int f(x)\left(\frac{1}{n^2}\hat{p}_{u}'''(x) + (4 - x^2)\hat{p}_{u}'(x) + x \hat{p}_{u}(x)\right)dx = 0.
    \end{equation*}
    Integrating by parts we obtain
    \begin{equation*}
        \frac{1}{n^2}\int f'''(x)\hat{p}_{u}dx = \int \left((x^2 - 4)f'(x)+3xf(x)\right)\hat{p}_{u}dx,
    \end{equation*}
    then, using that $p_{u}$ is a probability measure and (\ref{HaagerupEquation}), we have (\ref{HaagerupRecurrent}).
\end{proof}

Since $T^k g$ is bounded and $\hat{p}_u$ is a probability measure, we easily get
\begin{corollary}[\cite{HT2}, Corollary 3.6]
    For any $k \in \mathbb{N}$ and any function $g \in C_b^{\infty}$ we have
    \begin{equation}\label{GUEasymp}
        \int g(x)\hat{p}_{u}dx = \frac{1}{2\pi}\sum_{j=0}^{k-1}\frac{1}{n^{2j}}\int_2^2 [T^j g](t)\sqrt{4-t^2}dt + O(n^{-2k}).
    \end{equation}
\end{corollary}

O. Kopelevitch proved \cite{K} that the expansion \eqref{GUEasymp} is convergent for certain class of entire functions. 

\begin{theorem}[\cite{K}, Theorem 1]\label{KopelevitchTheorem}
    Let $g$ be an entire function of order two and finite type $\sigma_g$. Then for GUE random matrix $X_n$ of order $n \ge n_{u}(\sigma_g)$, one has the following convergent expansion:
    \begin{equation}
        \frac{1}{n}\mathbb{E}\{\Tr(g(X_n))\} = \int g(x)\hat{p}_{u}dx = \frac{1}{2\pi}\sum_{j=0}^{\infty}\frac{1}{n^{2j}}\int_2^2 [T^j g](t)\sqrt{4-t^2}dt.
    \end{equation}
    
    For an entire function $g$ of order less than 2, the expansion is convergent for any $n \ge 1$. 
\end{theorem}

We will prove similar theorems for GSE and GOE cases. Our results work for the same set of functions, but with additional restrictions on the set of $n$'s even for functions of order less than 2. 

\begin{theorem}\label{thm:GSE-conv}
    Let $g$ be an entire function of order $\rho_g\le 2$ and type $\sigma_g$ with $\sigma_g<+\infty$ if $\rho_g=2$. Then for GSE random matrix $X_n$ of order $n \ge n_{s}(\rho_g,\sigma_g)$, one has the following convergent expansion:
    \begin{equation}
        \int g(x)\hat{p}_{s}dx = \frac{3}{4\pi}\sum_{j=0}^{\infty}\frac{1}{n^{2j}}\int_2^2 \biggl[T^j \biggl(\frac{1}{n}f'-xf\biggr)\biggr](t)\sqrt{4-t^2}dt,
    \end{equation}
    where $f$ is any solution of the differential equation
    \begin{equation}\label{DiffGSE}
        g(x) = -\frac{4}{n^2}f'''(x) + \left(x^2 - \frac{8n+2}{n}\right)f'(x).
    \end{equation}
\end{theorem}

\begin{theorem}\label{thm:GOE-conv}
    Let $g$ be an entire function of order $\rho_g\le 2$ and type $\sigma_g$ with $\sigma_g<+\infty$ if $\rho_g=2$. Then for GOE random matrix $X_n$ of order $n \ge n_{o}(\rho_g,\sigma_g)$, one has the following convergent expansion:
    \begin{equation}
        \int g(x)\hat{p}_{o}dx = \frac{3}{4\pi}\sum_{j=0}^{\infty}\frac{1}{n^{2j}}\int_2^2 \biggl[T^j \biggl(\frac{1}{n}f'-xf\biggr)\biggr](t)\sqrt{4-t^2}dt,
    \end{equation}
    where $f$ is any solution of the differential equation
    \begin{equation}\label{DiffGOE}
        g(x) = -\frac{4}{n^2}f'''(x) + \left(x^2 - \frac{4n-2}{n}\right)f'(x).
    \end{equation}
\end{theorem}

We will prove these theorems reducing them to Theorem~\ref{KopelevitchTheorem} by Kopelevitch.
\begin{proof}[Proof of Theorem~\ref{thm:GSE-conv}.] Multiply (\ref{GSEdiff}) for $\sigma^2 = 1/n$ by an entire function $f$ of order $\rho_f \leq 2$ and type $\sigma_f$ with $\sigma_f < \infty$ if $\rho_g = 2$, and integrate over the real line:
    \begin{equation}\label{stepDiff1}
        \int f(x)\left(\frac{4}{n^2} \hat{p}_{s}''' + \left(\frac{8n+2}{n}- x^2\right)\hat{p}_{s}' - 2x \hat{p}_{s}\right)dx = -\frac{3}{2}\int f(x)\left(\frac{1}{n}\hat{p}_{u}+x \hat{p}_{u}\right)dx.
    \end{equation}
As it was shown by  J.~Heittokangas,  R.~Korhonen, and J.~Rättyä~\cite{HKR}, any solution of the differential equation \eqref{DiffGSE} with an entire function $g$ of order less or equal to 2 and of a finite type, is also an entire function of order not more than 2 and of a finite type.

Because of the conditions on $f$, if $n$ is large enough, when integrating by parts, the terms outside of the integrals vanish, because the argument in exponents in $p_s(x)$ grows as $n$ grows. So we obtain
    \begin{equation}\label{stepDiff2}
        \int \left(-\frac{4}{n^2}f''' + \left(x^2 - \frac{8n+2}{n}\right)f'\right)\hat{p}_{o}dx = \frac{3}{2}\int \left(\frac{1}{n}f'-x f\right)\hat{p}_{u}dx.
    \end{equation}

    Now, using (\ref{DiffGSE}), (\ref{GUEasymp}) and the fact that order and type are invariant under taking derivatives and multiplying by constants and $x^n$, one can see that the function $\frac{1}{n}f' - x f$
    satisfies all conditions of Kopelevitch's theorem.
\end{proof}
\begin{proof}
    Proof of Theorem~\ref{thm:GOE-conv} is exactly the same with following formulas instead of \eqref{stepDiff1} and \eqref{stepDiff2} 
    \begin{equation}
            \int f(x)\left(\frac{4}{n^2} \hat{p}_{s}''' + \left(\frac{8n+2}{n}- x^2\right)\hat{p}_{s}' - 2x \hat{p}_{s}\right)dx = -\frac{3}{2}\int f(x)\left(\frac{1}{n}\hat{p}_{u}+x \hat{p}_{u}\right)dx.
        \end{equation}
    \begin{equation}
            \int \left(-\frac{4}{n^2}f''' + \left(x^2 - \frac{8n+2}{n}\right)f'\right)\hat{p}_{o}dx = \frac{3}{2}\int \left(\frac{1}{n}f'-x f\right)\hat{p}_{u}dx.\qedhere
    \end{equation}
\end{proof}

\section{Differential equations on moment generating functions}\label{SectionDifferentialMoments}

Here we derive differential equations on the moment generating functions for Gaussian Ensembles. 
These results are due to Ledoux \cite{L}, but we obtain them by a different technique.

Let us denote the moment generating function (equivalently, the bilateral Laplace transform) as
\begin{equation*}
    (\mathcal{L}\mu)(s) = \int_{\mathbb{R}}\exp(s x)\mu(x)dx.
\end{equation*}

\subsection{GUE case}
\begin{theorem}
The moment generating function of GUE (\ref{GUEdensity}) satisfies the differential equation
\begin{equation}
    s(\mathcal{L}p_u)''(s) - (\mathcal{L}p_u)'(s) - s(\sigma^4 s^2 + 4n\sigma^2)(\mathcal{L}p_u)(s) = 0.
\end{equation}    
\end{theorem}
\begin{proof}
    Multiply (\ref{GUEdiff}) by $\exp(s x)$ and integrate over the real line:
    \begin{equation*}
        \int \exp(s x)\left(\sigma^4p_{u}'''(x) + (4n\sigma^2 - x^2)p_{u}'(x) + x p_{u}(x)\right)dx = 0.
    \end{equation*}
    Integrating by parts we obtain
    \begin{equation*}
         s\int x^2\exp(s x)p_{s}dx - \int x\exp(s x)p_{s}dx + s(\sigma^4 s^2 + 4n\sigma^2)\int \exp(s x)p_{s}dx = 0,
    \end{equation*}
    \begin{equation*}
        s(\mathcal{L}p_u)''(s) - (\mathcal{L}p_u)'(s) - s(\sigma^4 s^2 + 4n\sigma^2)(\mathcal{L}p_u)(s) = 0.\qedhere
    \end{equation*}
\end{proof}
\subsection{GSE case}
\begin{theorem}
The moment generating function of GSE (\ref{GSEdensity}) satisfies the differential equation
\begin{equation}
    s(\mathcal{L}p_s)''(s) - s(4\sigma^4 s^2 + \sigma^2(8n+2))(\mathcal{L}p_s)(s) = -3(\mathcal{L}p_u)'(s) + 3\sigma^2 s (\mathcal{L}p_u)(s).
\end{equation}    
\end{theorem}
\begin{proof}
    Multiply (\ref{GSEdiff}) by $\exp(s x)$ and integrate over the real line:
    \begin{multline*}
        \int \exp(s x)\left(4\sigma^4p_{s}'''(x) + (\sigma^2(8n+2) - x^2)p_{s}'(x) - 2x p_{s}(x)\right)dx = \\
        = \int \exp(s x)\left(-3\sigma^2p_u' - 3x p_{u}(x)\right)dx,
    \end{multline*}
    then integrate by parts to obtain
    \begin{multline*}
        s\int x^2\exp(s x)p_{s}dx - s(4\sigma^4 s^2 + \sigma^2(8n+2))\int \exp(s x)p_{s}dx = \\
        = 3\sigma^2 s\int \exp(s x)p_{u}dx - 3\int x\exp(s x)p_{u}dx,
    \end{multline*}
    \begin{equation*}
        s(\mathcal{L}p_s)''(s) - s(4\sigma^4 s^2 + \sigma^2(8n+2))(\mathcal{L}p_s)(s) = -3(\mathcal{L}p_u)'(s) + 3\sigma^2 s (\mathcal{L}p_u)(s)\qedhere
    \end{equation*}
\end{proof}
\subsection{GOE case}
\begin{theorem}
The moment generating function of GOE (\ref{GOEdensity}) satisfies the differential equation
\begin{equation}
    s(\mathcal{L}p_o)''(s) - s(4\sigma^4 s^2 + \sigma^2(4n-2))(\mathcal{L}p_o)(s) = -3(\mathcal{L}p_u)'(s) + 3\sigma^2 s (\mathcal{L}p_u)(s).
\end{equation}    
\end{theorem}
\begin{proof}
    Again, multiply (\ref{GOEdiff}) by $\exp(s x)$ and integrate over the real line:
    \begin{multline*}
        \int \exp(s x)\left(4\sigma^4p_{o}'''(x) + (\sigma^2(4n-2) - x^2)p_{o}'(x) - 2x p_{o}(x)\right)dx = \\
        = \int \exp(s x)\left(-3\sigma^2p_u' - 3x p_{u}(x)\right)dx,
    \end{multline*}
    and integration by parts gives that
    \begin{multline*}
        s\int x^2\exp(s x)p_{o}dx - s(4\sigma^4 s^2 + \sigma^2(4n-22))\int \exp(s x)p_{o}dx = \\
        = 3\sigma^2 s\int \exp(s x)p_{u}dx - 3\int x\exp(s x)p_{u}dx,
    \end{multline*}
    \begin{equation*}
        s(\mathcal{L}p_o)''(s) - s(4\sigma^4 s^2 + \sigma^2(4n-2))(\mathcal{L}p_o)(s) = -3(\mathcal{L}p_u)'(s) + 3\sigma^2 s (\mathcal{L}p_u)(s).\qedhere
    \end{equation*}
\end{proof}
\section{The moment generating functions}\label{SectionMoment}
In this section we calculate explicit expressions of moment generating functions for Gaussian Ensembles. GUE case was done by U. Haagerup and S.Thorbjornsen in \cite{HT} (for reader's convenience, we reproduce their proof of Theorem~\ref{GUEtheorem} here). Then we prove Theorem \ref{GSEtheorem} and Theorem \ref{GOEtheorem} for Symplectic and Orthogonal cases, respectively.

\begin{theorem}[\cite{HT}, Theorem 2.5]\label{GUEtheorem}
    The moment generating function for a GUE random matrix $A_n$ of order $n$ with variance $\sigma^2$ equals
    \begin{equation}\label{GUEgen}
        (\mathcal{L}p_{GUE,n,\sigma^2})(s) = \mathbb{E}(\Tr(e^{sA_n})) = n\cdot e^{\frac{\sigma^2s^2}{2}}{_1F_1}\left(1-n;2;-\sigma^2s^2 \right).
    \end{equation}
\end{theorem}
\begin{theorem}\label{GSEtheorem}
    The moment generating function for a GSE random matrix $A_n$ of order $n$ with variance $\sigma^2$ equals
    \begin{multline}\label{GSEgen}
        (\mathcal{L}p_{GSE,n,\sigma^2})(s) = \mathbb{E}(\Tr (e^{sA_n})) = -e^{\frac{\sigma^2s^2}{2}}\sum_{i=0}^{n}\frac{\sigma^{2i}s^{2i}2^i n!}{(2i)!(n-i)!}{_1F_1}\left(-(2n-2i);1+2i;-\sigma^{2}s^{2}\right) + \\ + n\cdot e^{\frac{\sigma^2s^2}{2}}{_1F_1}\left(1-n;2;-\sigma^2s^2 \right).
    \end{multline}
\end{theorem}
\begin{theorem}\label{GOEtheorem}
    The moment generating function for a GOE random matrix $A_n$ of order $n$ with variance $\sigma^2$ equals
    \begin{multline}\label{GOEgen}
        (\mathcal{L}p_{GOE,n,\sigma^2})(s) = \mathbb{E}(\Tr (e^{sA_n})) = \mathcal{L}(p_{GUE,n,\sigma^2}) + \mathbb{I}_{\text{n is odd}} e^{\sigma^2s^2}{_1F_1}\left(-\frac{n-1}{2};\frac{1}{2};-2\sigma^2s^2\right) - \\
        -e^{\frac{\sigma^2s^2}{2}}\sum_{i=0}^{\floor*{\frac{n-1}{2}}}\frac{\sigma^{2i}s^{2i}(n-1)!!}{(2i)!(n-1-2i)!!}{_1F_1}\left(-(n-1-2i);1+2i;-\sigma^{2}s^{2}\right) + \\
        + \mathbb{I}_{\text{n is even}}\frac{(n-1)!!}{2^{\frac{n}{2}} (n-1)!}\biggl[2s\frac{(n-1)!}{\left(\frac{n-2}{2}\right)!}{_1F_1}\left(-\frac{n-2}{2}, \frac{3}{2}, -2\sigma^2 s^2\right)e^{\sigma^2 s^2}\sqrt{2}\int_0^{\sigma s} e^{-\frac{t^2}{2}}dt  + \\
        + e^{\frac{\sigma^2s^2}{2}}\sum_{j=1}^{n-1}\binom{n-1}{j}(-i)^{n-1-j}H_{n-1-j}(i \sigma \sqrt{2}s)H_{j-1}\left(-\frac{\sigma s}{\sqrt{2}}\right) \biggr].
    \end{multline}
\end{theorem}

We start with a couple of preparatory calculations.
\begin{lemma}\label{OriginalTwoHermites}
    For any $k, n \in \mathbb{N}$ such that $k \leq n$,
    \begin{equation}
        \int_{\mathbb{R}}e^{sx-x^2}H_n(x)H_k(x)dx = \frac{\sqrt{\pi}n!2^k}{(n-k)!}e^{\frac{s^2}{4}}s^{n-k}{_1F_1}\left(-k,n-k+1,-\frac{s^2}{2}\right).
    \end{equation}
\end{lemma}
\begin{proof}
    By the substitution $y = x-\frac{s}{2}$ and properties (\ref{orthRel}), (\ref{HermiteShift}) of Hermite polynomials we have
    \begin{gather*}
        \int_{\mathbb{R}}e^{sx-x^2}H_n(x)H_k(x)dx = e^{\frac{s^2}{4}}\int_{\mathbb{R}}e^{-y^2}H_n\left(y+\frac{s}{2}\right)H_k\left(y+\frac{s}{2}\right)dy = \\
        = e^{\frac{s^2}{4}}\sum_{i=0}^{min\{k,n\}}\binom{k}{i} \binom{n}{i} 2^i i!\sqrt{\pi} s^{k+n-2i}.
    \end{gather*}
    Substituting $i$ by $k-i$ and using the definition of the confluent hypergeometric function (\ref{Hyper}), we get
    \begin{gather*}
        e^{\frac{s^2}{4}}\sum_{i=0}^{k}\binom{k}{i} \binom{n}{i} 2^i i!\sqrt{\pi} s^{k+n-2i} = \sqrt{\pi}e^{\frac{s^2}{4}}s^{n-k}\sum_{i=0}^{k}\frac{k!n!2^i s^{2k-2i}}{(k-i)!i!(n-i)!} = \\
        =\frac{\sqrt{\pi}n!2^k}{(n-k)!}e^{\frac{s^2}{4}}s^{n-k}\sum_{i=0}^{k}\frac{k(k-1)\dots(k-i+1)}{(n-k+1)(n-k+2)\dots(n-k+i)i!}\left(\frac{s^2}{2}\right)^i = \\
        =\frac{\sqrt{\pi}n!2^k}{(n-k)!}e^{\frac{s^2}{4}}s^{n-k}{_1F_1}\left(-k,n-k+1,-\frac{s^2}{2}\right).\qedhere
    \end{gather*}
\end{proof}

\begin{corollary}\label{twoHermites}
    For any $k, n \in \mathbb{N}$ such that $0 \leq k \leq n$,
    \begin{equation*}
        \int_{\mathbb{R}}e^{sx-x^2}H_n(x)H_{n-k}(x)dx = \frac{\sqrt{\pi}n!2^{n-k}}{k!}e^{\frac{s^2}{4}}s^k{_1F_1}\left(-(n-k),1+k,-\frac{s^2}{2}\right).
    \end{equation*}
\end{corollary}

\begin{lemma}
    For any $n \in \mathbb{N}$, $a \in \mathbb{R}$,
    \begin{align}\label{antiHermite}
        \int_{a}^{x}e^{-\frac{t^2}{2}}H_{n}(t)dt = {}& {-2}\sum_{i=0}^{\floor*{\frac{n-1}{2}}} \frac{2^i (n-1)!!}{(n-1-2i)!!}\bigl(e^{-\frac{t^2}{2}}H_{n-1-2i}(t)\bigr) \Big|_{t=a}^{t=x} + \\
        \notag
        &{}+ \mathbb{I}_{\text{n is even}}2^{\frac{n}{2}}(n-1)!!\int_a^x e^{-\frac{t^2}{2}}dt.
    \end{align}
\end{lemma}

\begin{proof}Property (\ref{recurrent}) of Hermite polynomials and integration by parts yield recursive relation
    \begin{multline*}
        \int_{a}^{x} e^{-\frac{t^2}{2}}H_{n}(t)dt = 2\int_{a}^{x} te^{-\frac{t^2}{2}}H_{n-1}(t)dt - \int_{a}^{x} e^{-\frac{t^2}{2}}H_{n-1}'(t)dt= \\
        -2e^{-\frac{t^2}{2}}H_{n-1}(t) \Big|_a^x + 2(n-1)\int_{a}^{x} e^{-\frac{t^2}{2}}H_{n-2}(t)dx,
    \end{multline*}
which implies~\eqref{antiHermite}.
\end{proof}

\begin{corollary}
    For odd $n$ we have
   \begin{equation}\label{antiHermiteOdd}
      \int_{-\infty}^{x}e^{-\frac{t^2}{2}}H_{n}(t)dt = -2\sum_{i=0}^{\floor*{\frac{n-1}{2}}} \frac{2^i (n-1)!!}{(n-1-2i)!!}e^{-\frac{x^2}{2}}H_{n-1-2i}(x).
   \end{equation}
\end{corollary}

\begin{corollary}
   For even $n$ because of (\ref{HermiteNumbers}) we have
   \begin{equation}\label{antiHermiteEven}
       \int_{0}^{x}e^{-\frac{t^2}{2}}H_{n}(t)dt = -2\sum_{i=0}^{\floor*{\frac{n-1}{2}}} \frac{2^i (n-1)!!}{(n-1-2i)!!}e^{-\frac{x^2}{2}}H_{n-1-2i}(x)+2^{\frac{n}{2}}(n-1)!!\int_a^x e^{-\frac{t^2}{2}}dt.
   \end{equation}
\end{corollary}

As $H_n(x)$ is an odd function for odd $n$, we have:
\begin{gather*}
    \frac{1}{2}\int_{\mathbb{R}} \sign(x-t)e^{-\frac{t^2}{2}}H_n(t)dt = \frac{1}{2}\left[\int_{-\infty}^{x} e^{-\frac{t^2}{2}}H_n(t)dt - \int_{x}^{+\infty} e^{-\frac{t^2}{2}}H_n(t)dt \right]= \\
    = \frac{1}{2}\left[\int_{-\infty}^{x} e^{-\frac{t^2}{2}}H_n(t)dt + \int_{-\infty}^{-x} e^{-\frac{t^2}{2}}H_n(t)dt \right] = \int_{-\infty}^{x} e^{-\frac{t^2}{2}}H_n(t)dt.
\end{gather*}

Analogously, $H_n(x)$ is an even function for even $n$, thus
\begin{multline*}
    \frac{1}{2}\int_{\mathbb{R}} \sign(x-t)e^{-\frac{t^2}{2}}H_n(t)dt = \frac{1}{2}\left[\int_{-\infty}^{x} e^{-\frac{t^2}{2}}H_n(t)dt - \int_{x}^{+\infty} e^{-\frac{t^2}{2}}H_n(t)dt \right]= \\
    =\frac{1}{2}\left[\int_{-\infty}^{x} e^{-\frac{t^2}{2}}H_n(t)dt - \int_{-\infty}^{-x} e^{-\frac{t^2}{2}}H_n(t)dt \right] = \int_{0}^{x} e^{-\frac{t^2}{2}}H_n(t)dt.
\end{multline*}

Therefore, according to (\ref{antiHermiteOdd}) and (\ref{antiHermiteEven}), we have
\begin{align}\label{middleSimple}
    \frac{1}{2}\int_{\mathbb{R}} \sign(x-t)e^{-\frac{t^2}{2}}H_n(t)dt ={}& -2\sum_{i=0}^{\floor*{\frac{n-1}{2}}} \frac{2^i (n-1)!!}{(n-1-2i)!}e^{-\frac{x^2}{2}}H_{n-1-2i}(x) + \\
    \notag &{}+ \mathbb{I}_{\text{$n$ is even}}2^{\frac{n}{2}}(n-1)!!\int_a^x e^{-\frac{t^2}{2}}dt.
\end{align}

Now we can use (\ref{middleSimple}) and Corollary \ref{twoHermites} to calculate the next integral that is a part of the formula (\ref{GUEdensity}) and, therefore, (\ref{GSEdensity}) and (\ref{GOEdensity}):
\begin{multline}\label{supportCalc}
    \sqrt{\frac{n}{2}}\int_{\mathbb{R}}e^{sx}\phi_{n-1}(x)\frac{1}{2}\int_{\mathbb{R}} \sign(x-t)\phi_{n}(t)dtdx = \\
    =\frac{1}{\sqrt{\pi}2^n(n-1)!}\int_{\mathbb{R}}e^{sx-\frac{x^2}{2}}H_{n-1}(x)\frac{1}{2}\int_{\mathbb{R}} \sign(x-t)e^{-\frac{t^2}{2}}H_{n}(t)dtdx = \\
    \shoveright{=-e^{\frac{s^2}{4}}\sum_{i=0}^{\floor*{\frac{n-1}{2}}}\frac{s^{2i}(n-1)!!}{2^i(2i)!(n-1-2i)!!}{_1F_1}\left(-(n-1-2i);1+2i;-\frac{s^2}{2}\right) +} \\
    {}+ \mathbb{I}_{\text{n is even}}\frac{2^{\frac{n}{2}}(n-1)!!}{\sqrt{\pi}2^n (n-1)!}\int_{\mathbb{R}}e^{s x-\frac{x^2}{2}}H_{n-1}(x)\int_0^x e^{-\frac{t^2}{2}}dtdx.
\end{multline}
\subsection{GUE case}
For reader's convenience, we reproduce here the proof for the GUE case from \cite{HT}.

For simplicity of computations, let us take $\sigma^2 = 1/2$. Then the GUE density (\ref{GUEdensity}) equals
\begin{equation*}
    p_{GUE,n,1/2}(x) = \sum_{k=0}^{n-1}\phi_k(x)^2.
\end{equation*}
Using (\ref{recurrent}) and Corollary \ref{twoHermites}, we obtain the Laplace transform of the GUE density:
\begin{multline*}
    \int_{\mathbb{R}} e^{s x}\left(\sum_{i=0}^{n-1}\phi_i(x)^2\right)dx = \frac{\sqrt{2n}}{s}\int_{\mathbb{R}} e^{s x}\phi_n(x)\phi_{n-1}(x)dx =\\
    = \frac{1}{2^{n-1}(n-1)!\sqrt{\pi}s}\int_{\mathbb{R}} e^{s x-x^2}H_n(x)H_{n-1}(x)dx = n e^{\frac{s^2}{4}}{_1F_1}\left(1-n,2,-\frac{s^2}{2}\right)
\end{multline*}
for $s \neq 0$. For $s = 0$ it follows from the orthogonality relations (\ref{orthRel}).

The general case (\ref{GUEgen}) is obtained by a change of variables $s\mapsto\sqrt{2\sigma}s$.

\subsection{GSE case}
Again, let us take $\sigma^2 = 1/2$, so the GSE density (\ref{GSEdensity}) equals
\begin{equation*}
    2p_{GSE,n}(x) = p_{GUE,2n+1}(x)+\sqrt{\frac{2n+1}{2}}\phi_{2n}(x)\frac{1}{2}\int_{\mathbb{R}} \sign(x-t)\phi_{2n+1}(t)dt.
\end{equation*}

We already know the Laplace transform of the first term. For the second term we use  (\ref{supportCalc}) with $2n+1$ in place of $n$, so this term equals
\begin{multline*}
    -e^{\frac{s^2}{4}}\sum_{i=0}^{n}\frac{s^{2i}(2n)!!}{2^i(2i)!(2n-2i)!!}{_1F_1}\left(-(2n-2i);1+2i;-\frac{s^2}{2}\right) ={}\\
    = -e^{\frac{s^2}{4}}\sum_{i=0}^{n}\frac{s^{2i}n!}{(2i)!(n-i)!}{_1F_1}\left(-(2n-2i);1+2i;-\frac{s^2}{2}\right),
\end{multline*}
and (\ref{GSEgen}) is proven for $\sigma^2=1/2$. As usual, the general case follows by change of variable $s \mapsto \sqrt{2}\sigma s$.

\subsection{GOE case}
For $\sigma^2 = 1/2$ the density  (\ref{GOEdensity}) equals
\begin{equation}\label{GOEdensitySimplifiedOdd}
    p_{GOE,n,\sigma^2}(x) = 2p_{GSE, \frac{n-1}{2}}(x)+\mathbb{I}_{\text{n is odd}}\phi_{n-1}(x)\left(\int_{\mathbb{R}}\phi_{n-1}(t)dt)\right)^{-1}.
\end{equation}
Since $e^{-\frac{x^2}{2}} H_n(x)$ is an eigenfunction of Fourier transform with the eigenvalue $(-i)^n$, we readily obtain the Laplace transform (that is, integrating with the weight $e^{s x}$) of the last term in (\ref{GOEdensitySimplifiedOdd}) (cf. (\ref{HermiteHyper}))
\begin{equation*}
    \frac{\int_{\mathbb{R}}e^{sx}\phi_{n-1}(x)dx}{\int_{\mathbb{R}}\phi_{n-1}(t)dt} = \frac{(-i)^n \sqrt{2 \pi}\phi_{n-1}(i s)}{(-i)^n \sqrt{2 \pi}\phi_{n-1}(0)} = e^{\frac{s^2}{2}}{_1F_1}\left(-\frac{n-1}{2};\frac{1}{2};-s^2\right).
\end{equation*}
The Laplace transform of $2p_{GSE, \frac{n-1}{2}}$ is obtained by (\ref{supportCalc}), and it equals
\[-e^{\frac{s^2}{4}}\sum_{i=0}^{\floor*{\frac{n-1}{2}}}\frac{s^{2i}(n-1)!!}{2^i(2i)!(n-1-2i)!!}{_1F_1}\left(-(n-1-2i);1+2i;-\frac{s^2}{2}\right) + \]
\[+ \mathbb{I}_{\text{n is even}}\frac{2^{\frac{n}{2}}(n-1)!!}{\sqrt{\pi}2^n (n-1)!}\int_{\mathbb{R}}e^{s x-\frac{x^2}{2}}H_{n-1}(x)\int_0^x e^{-\frac{t^2}{2}}dtdx.\]
Denote
\begin{equation*}
   \mathcal{A}(x) = \sum_{n=0}^{\infty}\frac{A_n(s)}{n!}x^n, \quad A_n(s) = \int_{\mathbb{R}}e^{s x-\frac{x^2}{2}}H_{n}(x)\int_0^x e^{-\frac{t^2}{2}}dtdx.
\end{equation*}
By recurrence relations (\ref{recurrent}), (\ref{OriginalTwoHermites}) and integration by parts, we get
\begin{multline*}
   A_{n+1}(s) = \int_{\mathbb{R}}e^{s x-\frac{x^2}{2}}H_{n+1}(x)\int_0^x e^{-\frac{t^2}{2}}dtdx = \\
   = 2\int_{\mathbb{R}}xe^{s x-\frac{x^2}{2}}H_{n}(x)\int_0^x e^{-\frac{t^2}{2}}dtdx + \int_{\mathbb{R}}e^{s x-\frac{x^2}{2}}H_{n}'(x)\int_0^x e^{-\frac{t^2}{2}}dtdx = \\
   = 2sA_n(s) + 2nA_{n-1}(s) + \sqrt{\pi}e^{\frac{s^2}{4}}s^n,
\end{multline*}
hence
\begin{equation*}
   \mathcal{A}'(x) = (2s+2x)\mathcal{A}(x) + \sqrt{\pi}e^{\frac{s^2}{4}}e^{sx}.
\end{equation*}
Solving this ordinary differential equation of the first order, we get
\begin{equation*}
   \mathcal{A}(x) = e^{2sx+x^2}\left(C(s) + \sqrt{\pi}e^{\frac{s^2}{4}}\int_0^x e^{-st-t^2}dt\right).
\end{equation*}
Further, 
\begin{equation*}
   C(s) = \mathcal{A}(0) = A_0(s) = \int_{-\infty}^{\infty}e^{s x-\frac{x^2}{2}}\int_0^x e^{-\frac{t^2}{2}}dtdx = \sqrt{2 \pi} e^{\frac{s^2}{2}} \cdot \int_0^{\frac{s}{\sqrt{2}}} e^{-\frac{t^2}{2}}dt.
\end{equation*}
Finally, using the generating function of the Hermite polynomials (\ref{HermiteGenerating}), relation (\ref{HermiteHyper}) and general Leibniz rule, we obtain
\begin{multline}
   A_{n-1}(s) = \mathcal{A}^{(n-1)}(0) = 2s\frac{(n-1)!}{\left(\frac{n-2}{2}\right)!}{_1F_1}\left(-\frac{n-2}{2}, \frac{3}{2}, -s^2\right)e^{\frac{s^2}{2}}\sqrt{2\pi}\int_0^{\frac{s}{\sqrt{2}}} e^{-\frac{t^2}{2}}dt+\\
   +\sqrt{\pi}e^{\frac{s^2}{4}}\sum_{j=1}^{n-1}\binom{n-1}{j}(-i)^{n-1-j}H_{n-1-j}(i s)H_{j-1}\left(-\frac{s}{2}\right).
\end{multline}
Now, collecting all terms, we have proven (\ref{GOEgen}) in the case $\sigma^2 = 1/2$. The general case follows by change of variables $s \to \sigma\sqrt{2}s$.
\section{Asymptotic expansion of the moment generating functions}\label{SectionMomentExpansion}
We give expansions only for GUE and GSE cases. It is possible to get the expansion for the GOE case, but expressions are too cumbersome to be included. 
\subsection{GUE case}
We can start with (\ref{GUEgen}):
\begin{equation}\label{GUEgenNormalized}
    \mathcal{L}(p_{GUE,n,\sigma^2}) = n\cdot e^{\frac{s^2}{2n}}{_1F_1}\left(1-n;2;-\frac{s^2}{n}\right)
\end{equation}
As it was shown in \cite{K}, this expression has the following expansion:
\begin{equation*}
n\cdot e^{\frac{s^2}{2n}}{_1F_1}\left(1-n;2;-\frac{s^2}{n}\right) = e^{\frac{s^2}{2n}}\sum_{i=0}^{n-1}\frac{n(n-1)\dots(n-i)}{i!(i+1)!}\frac{s^{2i}}{n^i}.
\end{equation*}
Let us recall the definition of the Stirling numbers of the first kind $\stirlingone{n}{k}$,
\begin{align*}
    x(x+1)\dots(x+n+1) &{}= \sum_{k=0}^n \stirlingone{n}{n-k}x^{n-k},\\
    x(x-1)\dots(x-n+1) &{}= \sum_{k=0}^n (-1)^k\stirlingone{n}{n-k}x^{n-k}.
\end{align*}
Then, changing the order of summation, we have
\begin{multline*}
    n\cdot e^{\frac{s^2}{2n}}{_1F_1}\left(1-n;2;-\frac{s^2}{n}\right) = e^{\frac{s^2}{2n}}\sum_{i=0}^{n-1}\frac{\sum_{j=0}^{i+1} (-1)^j\stirlingone{i+1}{i+1-j}n^{i+1-j}}{i!(i+1)!}\frac{s^{2i}}{n^i}=\\
    =e^{\frac{s^2}{2n}}\sum_{j=0}^n \left(\sum_{i=j-1}^{n-1} \frac{(-1)^j\stirlingone{i+1}{i+1-j}s^{2i}}{i!(i+1)!} \right)\frac{1}{n^{j-1}},
\end{multline*}
where $\stirlingone{0}{0} = 0$. So we can multiply the expansion of the exponent and the polynomial to get $1/n$-expansion of the Laplace transform of the GUE density.

\subsection{GSE case}
Lets us rewrite (\ref{GSEgen}) with $\sigma^2 = 1/2$:
\begin{equation}\label{GSEgenNormalized}
    \mathcal{L}(p_{GSE,n}) = \mathcal{L}(p_{GUE, n})-e^{\frac{s^2}{2n}}\sum_{i=0}^{n}\frac{s^{2i}2^i n!}{n^i(2i)!(n-i)!}{_1F_1}\left(-(2n-2i);1+2i;-\frac{s^2}{n}\right).
\end{equation}

The $1/n$-expansion of the first term is obtained in the previous subsection. Now let us calculate $1/n$-expansion of the second term:

\begin{multline*}
    e^{\frac{s^2}{2n}}\sum_{i=0}^{n}\frac{s^{2i}2^i n!}{n^i(2i)!(n-i)!}{_1F_1}\left(-(2n-2i);1+2i;-\frac{s^2}{n}\right) = \\
    =e^{\frac{s^2}{2n}}\sum_{i=0}^{n}\sum_{j=0}^{2n-2i}\frac{2^in(n-1)\dots(n-i+1)(2n-2i)\dots(2n-2i+j-1)s^{2i+2j}}{(2i+j)!j!n^{i+j}}.
\end{multline*}

We can write the product as a polynomial on $n$:
\begin{equation*}
    \prod_{k=0}^{i-1}(n-k)\prod_{k=0}^{j-1}(2n-2i-k) = \sum_{k=0}^{i+j}b_k^{(i,j)}n^{k} = \sum_{k=0}^{i+j}b_{i+j-k}^{(i,j)}n^{i+j-k}.
\end{equation*}

Then we can write
\begin{multline*}
    e^{\frac{s^2}{2n}}\sum_{i=0}^{n}\sum_{j=0}^{2n-2i}\frac{2^in(n-1)\dots(n-i+1)(2n-2i)\dots(2n-2i+j-1)s^{2i+2j}}{(2i+j)!j!n^{i+j}} = \\
    = e^{\frac{s^2}{2n}}\sum_{i=0}^{n}\sum_{j=0}^{2n-2i}\frac{2^i\sum_{k=0}^{i+j}b_{i+j-k}^{(i,j)}n^{i+j-k}s^{2i+2j}}{(2i+j)!j!n^{i+j}}.
\end{multline*}

Changing the order of summation, we get
\begin{equation*}
    e^{\frac{s^2}{2n}}\sum_{i=0}^{n}\frac{s^{2i}2^i n!}{n^i(2i)!(n-i)!}{_1F_1}\left(-(2n-2i);1+2i;-\frac{s^2}{n}\right) = e^{\frac{s^2}{2n}}\sum_{k=0}^{2n}\left(\sum_{i=0}^n \sum_{j=0}^{2n-2i}\frac{2^i b_{i+j-k}^{(i,j)}s^{2i+2j}}{(2i+j)!j!}\right)\frac{1}{n^k}.
\end{equation*}

\appendix

\section{Some properties of Hermite polynomials}

As above, \emph{Hermite functions} $\phi_k(x)$ and \emph{Hermite polynomials} $H_k(x)$ are defined as \begin{equation*}
    \phi_k(x) := \frac{1}{\sqrt{2^k k! \sqrt{\pi}}}H_k(x)e^{-\frac{x^2}{2}},\quad 
    H_k(x) := (-1)^k e^{x^2} \times \left( \frac{d^k}{dx^k}e^{-x^2} \right).
\end{equation*}
The polynomials $H_k$ are orthogonal in $L^2(\mathbb{R},e^{-x^2}\,dx)$, while
$\phi_k$ are orthonormal in $L^2(\mathbb{R},dx)$:
\begin{equation}\label{orthRel}
\begin{aligned}
    \int_{\mathbb{R}} H_k(x)H_l(x)e^{-x^2}dx &=
    \begin{cases}
        \sqrt{\pi}2^k k!,& k = l,\\
        0,              & k \neq l,
    \end{cases}\\
    \int_{\mathbb{R}} \phi_k(x)\phi_l(x)dx &=
    \begin{cases}
        1,& k = l,\\
        0,& k \neq l.\\
\end{cases}
\end{aligned}
\end{equation}
Clearly, $\phi_k(x)$ and $H_k(x)$ are odd functions when $k$ is odd, and even functions when $k$ is even. Also,
\begin{equation}\label{recurrent}
    H'_{n}(x)= 2xH_n(x) - H_{n+1}(x) = 2nH_{n-1}(x),
\end{equation}
thus
\begin{equation*}
    \phi_n'(x) = \sqrt{\frac{n}{2}}\phi_{n-1}(x) - \sqrt{\frac{n+1}{2}}\phi_{n+1}(x)
\end{equation*}
and
\begin{equation}\label{DerivSumSq}
    \frac{d}{dx}\left(\sum_{i=0}^{n-1}\phi_i(x)^2\right) = -\sqrt{2n}\phi_n(x)\phi_{n-1}(x).
\end{equation}
Values of Hermite polynomials at zero are called \emph{Hermite numbers}:
\begin{equation}\label{HermiteNumbers}
    H_n(0) =
    \begin{cases}
        0,& \mbox{if } n \mbox{ is odd}\\
        (-1)^{n/2} 2^{n/2} (n-1)!!, & \mbox{if } n \mbox{ is even}
    \end{cases}
\end{equation}
Hermite polynomials are represented by hypergeometric functions as
\begin{equation}\label{HermiteHyper}
\begin{aligned}
    H_{2n}(x) &= (-1)^{n}\,\frac{(2n)!}{n!} \,_1F_1\left(-n,\frac{1}{2};x^2\right),\\
   H_{2n+1}(x) &= (-1)^{n}\,\frac{(2n+1)!}{n!}\,2x \,_1F_1\left(-n,\frac{3}{2};x^2\right).
\end{aligned}
\end{equation}
The exponential generating function for Hermite polynomials is given by
\begin{equation}\label{HermiteGenerating}
    e^{2xt+t^2} = \sum_{n=0}^{\infty}\frac{H_n(x)}{n!}t^n,
\end{equation}
hence we immediately have 
\begin{equation}\label{HermiteShift}
    H_n(x + a) = \sum_{i=0}^{n}\binom{n}{i}H_i(x)(2a)^{n-i}.
\end{equation}

\section*{Acknowledgements}
I am deeply grateful to A.V.~Klimenko for helpful suggestions on
improving the presentation. I am deeply grateful to A.M.~Sodin for fruitful discussions. I am deeply grateful to A.I.~Bufetov for the proposed topic for work and his advice. During the work on this paper, I was supported by RFBR grant 16-31-00173, by the  Russian Academic Excellence Project '5-100' and Dobrushin Scholarship.

\bibliographystyle{alpha}

\end{document}